\documentclass{amsart}[12pt]

\newcommand{\hn}[1]{\h^{#1}}

\renewcommand{\S}{\Sigma}
\newcommand{\norma}[1]{\Vert #1 \Vert}
\newcommand{\ol}{\overline}

\newcommand{\cS}{\mathcal{S}}
\newcommand{\cQ}{\mathcal{Q}}

\newcommand{\es}{\emptyset}
\usepackage{enumerate,latexsym}
\usepackage{amsmath,amsthm,amsfonts,amssymb}
\usepackage[hidelinks]{hyperref}
\newcommand{\DOI}[1]{\href{https://doi.org/#1}{DOI:#1}}

\usepackage{soul}
\usepackage{ctable}
\usepackage{yfonts}
\usepackage{amssymb}
\usepackage{amsthm}
\usepackage{array}
\usepackage{booktabs}
\usepackage{hhline}
\usepackage{xy}
\usepackage{epsfig}
\usepackage{color}
\usepackage{upgreek}
\usepackage[english]{babel}
\usepackage{epigraph}
\usepackage{fancybox}
\usepackage{shadow}
\usepackage{afterpage}
\usepackage{mathrsfs}
\usepackage{enumitem}
\usepackage{tabularx}
\usepackage{subcaption}
\usepackage{graphicx}
\usepackage{type1cm}
\usepackage{eso-pic}
\usepackage{color}
\usepackage{upgreek}
\usepackage{verbatim}

\newtheorem{theorem}{Theorem}[section]

\newtheorem{lemma}[theorem]{Lemma}
\newtheorem{claim}[theorem]{Claim}
\theoremstyle{definition}
\newtheorem{definition}[theorem]{Definition}
\newtheorem{remark}[theorem]{Remark}
\newtheorem{example}[theorem]{Example}

\theoremstyle{plain}

\newcommand{\vt}{\vspace{.1cm}}

\newcommand{\R}{\mathbb{R} }

\newcommand{\N}{\mathbb{N} }
\newcommand{\h}{\mathbb{H} }

\newcommand{\s}{\mathbb{S}}

\renewcommand{\rho}{\varrho}
\newcommand{\abs}[1]{\vert #1\vert}
\renewcommand{\Theta}{\varTheta}
\renewcommand{\Lambda}{\varLambda}
\renewcommand{\Omega}{\varOmega}
\renewcommand{\Sigma}{\varSigma}
\renewcommand{\tau}{\uptau}
\usepackage{amsmath}% http://ctan.org/pkg/amsmath

\newcommand{\overbar}[1]{\mkern 1.5mu\overline{\mkern-1.5mu#1\mkern-1.5mu}\mkern 1.5mu}

\newcommand{\forma}[1]{\langle #1\rangle}

\makeatletter
\newcommand{\tpitchfork}{%
 \vbox{
 \baselineskip\z@skip
 \lineskip-.52ex
 \lineskiplimit\maxdimen
 \m@th
 \ialign{##\crcr\hidewidth\smash{$-$}\hidewidth\crcr$\pitchfork$\crcr}
 }%
}
\makeatother

\begin{document}

\title[Dynamical Stability of Translating Solitons]{Dynamical Stability of Translating
Solitons to Mean Curvature Flow in Hyperbolic Space}
\author[R.F. de Lima \and A.K. Ramos]{Ronaldo F. de Lima \and Álvaro K. Ramos}
\address[A1]{Departamento de Matem\'atica - UFRN}
\email{ronaldo.freire@ufrn.br}
\address[A2]{Departamento de Matemática Pura e Aplicada - UFRGS}
\email{alvaro.ramos@ufrgs.br}

\maketitle

\begin{abstract}
We develop the theory of translating solitons for the Mean
Curvature Flow (MCF) in hyperbolic space of dimension $n+1\ge 3$.
More specifically, we establish that horospheres
are dynamically stable as radial graphical solutions to MCF.
To that end, we construct rotationally invariant
translators analogous to the winglike solitons introduced by
Clutterbuck, Schnürer and Schulze, which serve as barriers in an
argument based on White’s avoidance principle and the
strong maximum principle for parabolic PDEs.

\end{abstract}

\section{Introduction}

A \emph{mean curvature flow} (MCF) in a Riemannian manifold $\overbar M$ is a one-parameter family
$\{
F_t\colon M\to\overbar M
\}_{t\in[0,T)}$
of immersions satisfying the evolution equation
\begin{equation} \label{eq:mcf}
\frac{\partial F}{\partial t}={\rm\bf H}_t,
\end{equation}
where ${\rm\bf H}_t$ is the mean curvature vector of $f_t$ with the metric induced by $\overbar M$.
In this context, we say that $F_t$ is a \emph{solution} to MCF.

A distinguished class of solutions to MCF is that of translating solitons.
Such solutions are generated by the Killing field defined by a one-parameter
group of translations along a geodesic in the ambient manifold $\overbar M$. The initial condition of
a translating soliton is then called a \emph{translator}.
Translating solitons in Riemannian manifolds constitute a main topic in the general theory
of extrinsic geometric flows (which includes MCF), and they have been studied from many points of view,
including construction, classification, asymptotic behaviour, stability and so on;
see, e.g.,~\cite{BL,schulze,DS, dLP,dLRS02,dLRS, EH, G, gs,H,I, LM, P1,P2,SS} and the references therein.

In Euclidean space, translators appear
as parabolic rescalings of type II singularities of certain solutions to mean curvature flow (cf.~\cite{HS}).
The best known are the cylinder over the
graph of the function $f(t) = -\log(\cos t),$ $t\in(-\pi/2,\pi/2)$, called the \emph{grim reaper};
the rotational entire graph over $\R^2$
obtained by Altschuler and Wu~\cite{altschuler}, known as the
\emph{bowl soliton} or \emph{translating paraboloid}, and the one-parameter family of rotational
annuli obtained by Clutterbuck, Schnürer and Schulze~\cite{schulze},
the so called \emph{winglike solutions} or \emph{translating catenoids}.

Another special type of solution to MCF in $\R^{n+1}=\R^n\times\R$, called \emph{graphical}, is the one in which
the images of the immersions are graphs over a fixed open set $U\subset\R^n\times\{0\}$. A translating soliton whose initial condition
is a graph, such as the grim reaper and the translating paraboloid, is a particular example of a graphical solution.
In~\cite{schulze}, the authors used their translating catenoids as good barriers to show that the translating paraboloid $\Sigma_0$
is dynamically stable in the sense that
any   perturbation of $\Sigma_0$ with $C^0$ decay is the initial condition of a graphical solution defined
for all $t>0$, which, as $t\to+\infty$, becomes asymptotic to the
translating soliton defined by $\Sigma_0$.

Inspired by the results in~\cite{schulze},
in the present work
we establish that horospheres in $\h^{n+1}$, which can be regarded
as the hyperbolic equivalent to the bowl solitons of $\R^{n+1}$,
are dynamically stable graphical solutions to MCF; see Theorem~\ref{thm-stability}.
For its proof, we first invoke a theorem by Unterberger~\cite{unterberger} which ensures
the existence of a graphical solution to MCF with given initial conditions. Then, we construct rotationally invariant
translating catenoids to be used
as barriers, allowing us to apply White's avoidance
principle~\cite{white} to prove convergence in space. Finally,
following the line of reasoning in~\cite{schulze},
we apply the strong maximum principle for parabolic PDE's
to prove convergence in time.

The paper is organized as follows. In Section~\ref{sec-preliminaries},
we briefly discuss some aspects of mean curvature flow in
hyperbolic space. In Section~\ref{sec-translators},
we construct the rotationally invariant translating solitons
to MCF, extending to the hyperbolic space $\hn{n+1}$,
in any dimension $n+1\geq3$, some of the results in~\cite{dLRS}, originally
established in $\h^{3}$. In Section~\ref{sec-stability},
we introduce the notion of radial graphic solution to MCF in $\h^{n+1}$, and then
prove Theorem~\ref{thm-stability}, concerning the stability
of horospheres in $\h^{n+1}$. Finally, in
Section~\ref{secopenprob}, we bring to the context of $\hn{n+1}$
the family of grim reapers in $\hn3$ constructed in~\cite{dLRS}, and
pose some questions regarding their stability.

\bigskip

\noindent
{\bf Acknowledgements.} We thank Lucas Ambrozio for his hospitality during the summer program at IMPA,
where this work was completed. We are also grateful to Barbara Nelli for useful remarks,
and especially for drawing our attention to Unterberger's paper~\cite{unterberger}.
A. Ramos was partially supported by CNPq - Brazil, grant number 406666/2023-7.

\section{Preliminaries} \label{sec-preliminaries}
Throughout the manuscript, we work with the upper half-space model for
the hyperbolic space $\h^{n+1}$ of dimension $n+1\ge 3$, that is,
$$\h^{n+1}:=(\R_+^{n+1},ds^2),$$
where $\R^{n+1}_+ = \{(x_1,\dots, x_{n+1})\in \R^{n+1}\mid x_{n+1}>0\},$
$ds^2:={d\bar s^2}/{x_{n+1}^2}$ and
$d\bar s^2$ is the standard Euclidean
metric of $\R_+^{n+1}.$ We will also use $\forma{\,,\,}$ to
denote the metric $ds^2$.

Let $\Sigma$ be an oriented hypersurface
of \,$\h^{n+1}$ with
unit normal $N$ and shape operator
$A$, so that
\[
Av=-\ol \nabla_vN, \,\, v\in T\Sigma,
\]
where $\ol \nabla$ is the Levi-Civita connection of $\h^{n+1}$ and
$T\Sigma$ is the tangent bundle of $\Sigma$.
The principal curvatures of $\Sigma,$ that is,
the eigenvalues of $A,$ will be denoted by
$k_1, \dots ,k_n$, and
the mean curvature $H$ of $\Sigma$ is expressed by
\[
H=\frac{k_1+\cdots +k_n}n\cdot
\]
The mean curvature vector of $\Sigma$ is
\begin{equation*}\mathbf{H}=HN,\end{equation*}
which is invariant under the choice of orientation
$N\to -N$ and satisfies $\norma{\mathbf{H}} = \abs{H}$.

Given an oriented hypersurface $\Sigma\subset\R_+^{n+1},$ let
$\overbar N=(\overbar N_1,\dots, \overbar N_{n+1})$
be a unit normal of $\Sigma$ with respect to
the Euclidean metric $d\bar s^2.$
It is easily checked that
\begin{align*}
N(p)=x_{n+1}\overbar N(p), \,\,\, p=(x_1,\dots, x_{n+1})\in\Sigma,
\end{align*}
defines a unit normal of $\Sigma$ with respect to the
hyperbolic metric $ds^2.$
With these choices of orientation, let
$\overbar H$ (resp.~$H$) denote the mean curvature of $\Sigma$ with respect to the Euclidean metric
(resp. the hyperbolic metric) on $\R_+^{n+1}$.
Then, $\overbar H$ and $H$ satisfy the following relation
(cf.~identity (2.1) in~\cite{gs}):
\begin{equation} \label{eq-MCrelation}
H(p)=x_{n+1}\overbar H(p)+\overbar N_{n+1}(p) \,\,\, \forall p=(x_1,\dots, x_{n+1})\in\Sigma.
\end{equation}

\subsection{Translators to mean curvature flow} \label{sec-MCF}
Let, for $t\in\R$,
$\Gamma_t\colon \hn{n+1}\to \hn{n+1}$ be given by
\begin{align}\label{eqtranslhyp}
\Gamma_t(p):= e^tp, \quad p\in\h^{n+1}.
\end{align}
Then, $\{\Gamma_t\mid t\in \R\}$
is a 1-parameter subgroup of isometries of $\h^{n+1}$,
which are called
\emph{hyperbolic translations} (along the
vertical geodesic determined by the $x_{n+1}$ axis of $\h^{n+1}$).
The Killing field generated by $\{\Gamma_t\}$ is simply
$\xi(p) = p$, where we are using the abuse of notation
\begin{equation*}
p=(x_1,\dots, x_{n+1})\in \h^{n+1} \leftrightarrow x_1\partial_{x_1}+ \cdots +x_{n+1}\partial_{x_{n+1}}\in T_p\h^{n+1}.
\end{equation*}

In this context, we say that an oriented hypersurface
$\Sigma\subset\h^{n+1}$ is a \emph{translator}
to MCF if it satisfies
\begin{equation} \label{eq-translatorH301}
H(p)=\langle p,N(p)\rangle \,\,\,\forall p\in\Sigma,
\end{equation}
where $H$ is the mean curvature of $\Sigma$ with respect to
the unit normal $N$.
For an isometric immersion $F\colon M\to\h^{n+1}$
such that
$\Sigma:=F(M)$ is a translator, one can verify that
$\{F_t\colon M \to \hn{n+1}\mid t\in\R\}$, where
\begin{align*}
F_t:=\Gamma_t\circ F,
\end{align*}
satisfies~\eqref{eq:mcf}, and hence is a solution to MCF (cf.~\cite{hungerbuhler-smoczyk}).

\begin{example}\label{eghorosp}
Let $\mathscr H_h$ be the horosphere of $\h^{n+1}$
at height $h>0,$ i.e.,
\[
\mathscr H_h=\{(x_1,\dots ,x_{n+1})\in\h^{n+1}\mid x_{n+1}=h\}.
\]
At any point $p\in\mathscr H_h,$ we have that
$H(p)=1$ and $N(p)=h\partial_{x_{n+1}}.$ Thus
\[
\langle p,N(p)\rangle=\frac1{h^2}h^2=1=H(p) \,\,\ \forall p\in\mathscr H_h,
\]
and hence $\mathscr H_h$ is a translator to MCF in $\h^{n+1}.$
\end{example}

\subsection{Graphical solutions to MCF in $\h^{n+1}$}
Let $\mathcal S\subset\R^{n+1}_+$ be the hyperbolic hyperplane defined as the
upper hemisphere of the unit Euclidean sphere centered at the origin.
For $T>0$, consider $u\in C^{\infty}(\mathcal S\times (0,T))
\cap C^{0}(\mathcal S\times [0,T))$; for $t\in [0,T)$,
$u_t\colon \cS\to \R$ will denote the function $u_t(x) = u(x,t)$.
Then, $u$ defines the following flow of entire
radial graphs over $\mathcal S$:
\begin{equation} \label{eq-map}
F(x,t):=e^{u_t(x)}x, \,\,\,\,\, (x,t)\in\mathcal S\times[0, T).
\end{equation}

The family $F_t:=F(.\,, t)$ is a solution to MCF
in $\h^{n+1}$ if and only if the function $u$ satisfies (see~\cite{LX,unterberger})
\begin{equation} \label{eq-MCFequation}
\frac{\partial u}{\partial t}(x,t)
=x_{n+1}\sqrt{1+\|\nabla u_t(x)\|^2}H(x,t),
\end{equation}
where $\|\cdot\|$ is the Euclidean norm in $\R^{n+1}$,
$\nabla u_t(x)$ is the gradient of $u_t$ at
the point $x=(x_1,\dots,x_{n+1})\in\mathcal S$
in the canonical metric of $\s^n$,
and $H(x,t)$ is the mean curvature of $F_t(\cS)$ at $x$
with respect to the hyperbolic metric, using the upwards pointing orientation.

\begin{definition}
A function $u\in C^{\infty}(\cS\times (0,T))\cap C^{0}(\cS\times [0,T))$
satisfying~\eqref{eq-MCFequation} will be called a \emph{graphical solution}
to MCF in $\h^{n+1}$, and the map $F$ defined in~\eqref{eq-map} will be called a
\emph{graphical} MCF in $\h^{n+1}$.
\end{definition}

Writing $H(x,t)$ in terms of $u$ gives that~\eqref{eq-MCFequation}
can be seen as a parabolic quasi-linear equation~\cite{unterberger}.
More precisely, set
$$\Lambda = \cS\times \R^{n+1}\times \R^{(n+1)\times(n+1)},$$
and denote an element of $\Lambda$ by the coordinates $(x,p,r)$,
where $x\in \cS$, $p\in \R^{n+1}$ and
$r \in \R^{(n+1)\times(n+1)}$.
There exists a function $Q\colon \Lambda\to \R$
defining an elliptic quasi-linear operator
\begin{equation}\label{eqopQ}
\mathcal Q(u) = Q(x,Du_t,D^2u_t)
\end{equation}
such that~\eqref{eq-MCFequation} becomes
\begin{equation} \label{eq-MCFequation02}
\frac{\partial u}{\partial t}(x,t)=\mathcal Q(u).
\end{equation}

See Section~3 of~\cite{unterberger} for the explicit expression
of $Q$ and also for the fact that~\eqref{eq-MCFequation02} is
a quasi-linear (nonuniform) parabolic PDE. From this
fact, Unterberger derived the following longtime existence result
for graphical MCF in $\h^{n+1}$.

\begin{theorem}{\cite[Theorem 0.1]{unterberger}} \label{thm-unterberger}
For any locally Lipschitz continuous function $u_0$ on $\mathcal S$,
there exists a graphical solution $u=u(x,t)$ to {\rm MCF} in $\h^{n+1}$ with initial condition $u_0$, which
is defined for all $t\ge 0$.
\end{theorem}

\section{Translating Catenoids in $\h^{n+1}$} \label{sec-translators}

In this section, we generalize the arguments of~\cite{dLRS} and
present a family of complete translators to MCF in $\hn{n+1}$,
to be called the {\em translating catenoids},
analogous to the {\em winglike solitons} of~\cite{schulze}. In
the proof of our main result, the translating catenoids will be used
as barriers to prove the stability of horospheres
as entire translating graphical solutions to MCF.

Choose a positive smooth function $\phi$ on an open interval
$I\subset(0,+\infty)$, and let $\Sigma$ be the hypersurface parameterized by the map
\begin{equation} \label{eq-parameterization}
X(\theta_1,\dots, \theta_{n-1},s)=
(s\varphi(\theta_1,\dots,\theta_{n-1}),\phi(s))\in\R^{n}\times\R_+,
\end{equation}
where $\varphi$ is a local parameterization of the unit sphere $\s^{n-1}\subset\R^n$.
We shall call $\Sigma$ the \emph{vertical rotational graph determined by} $\phi$.

\begin{lemma} \label{lem-rotationalODE01}
A vertical rotational graph determined by a smooth function $\phi$ is
a translator to {\rm MCF} in \,$\h^{n+1}$
if and only if $\phi$ satisfies the second order {\rm ODE}:
\begin{equation} \label{eq-EDOPsi}
\phi''=-\phi'(1+(\phi')^2)\left(\frac{ns}{\phi^2}+\frac{n-1}s\right).
\end{equation}
\end{lemma}
\begin{proof}
For a rotational graph $\Sigma$ parameterized by
$X$ as in~\eqref{eq-parameterization}, we have that its
Euclidean unit normal $\overbar N$ is
\[
\overbar N:=\rho(-\phi'\varphi,1),
\quad \rho:=\frac{1}{\sqrt{1+(\phi'))^2}}\cdot
\]

Then, a direct computation gives that, with this orientation,
the Euclidean mean curvature $\overbar H$ is the following function of $s$:
\[
\overbar H=\frac{\rho}{n}\left(\frac{\phi''}{1+(\phi')^2}+\frac{(n-1)\phi'}{s}\right).
\]
Thus, from~\eqref{eq-MCrelation}, the mean curvature $H$ of $\Sigma$ in $\h^{n+1}$ with respect to $N:=\phi\overbar N$ is
\begin{equation} \label{eq-Hrotationaltranslator}
H=\phi\overbar H+\overbar N_{n+1}
=\rho\left(
\frac{\phi}{n}\left(\frac{\phi''}{1+(\phi')^2}+\frac{(n-1)\phi'}{s}\right)+1\right).
\end{equation}
It is also easily seen that the equality
\begin{equation} \label{eq-Xeta}
\langle X,N\rangle=\frac{\rho}{\phi}(\phi-s\phi')
\end{equation}
holds everywhere on $\Sigma.$

From~\eqref{eq-Hrotationaltranslator} and~\eqref{eq-Xeta}, we conclude that
equation~\eqref{eq-translatorH301} for the vertical graph
$\Sigma$ is equivalent to the second order ODE:
\begin{equation*}
\phi''=-\phi'(1+(\phi')^2)\left(\frac{ns}{\phi^2}+\frac{n-1}s\right),
\end{equation*}
which proves the lemma.
\end{proof}

To construct the family of translating catenoids, besides rotational vertical graphs,
we also need to consider \emph{horizontal rotational graphs}. Given a smooth positive function
$d$ on an interval $(1-\delta,1+\delta)\subset (0,+\infty)$, such a graph is a rotational hypersurface
$\Sigma$ defined as
\[
\Sigma:=\{(x_1,\dots,x_{n+1})\in\h^{n+1}\mid x_1^2+\cdots+x_{n}^2=d^2(x_{n+1})\}.
\]

The next lemma characterizes horizontal rotational graphs that are translators.

\begin{lemma} \label{lem-rotationalODE021}
A horizontal rotational graph $\Sigma$ determined by a smooth function $d$ is
a translator to {\rm MCF} in $\h^{n+1}$ if and only if the function $d$
satisfies the {\rm ODE}:
\begin{equation} \label{eq-lemmahorizontalgraph}
d''=(1+(d')^2)\left(\frac{nd}{x_{n+1}^2}+\frac{n-1}{d}\right)\cdot
\end{equation}
In particular, such a solution $d$ is strictly convex.
\end{lemma}

\begin{proof}
Defining the function $\Phi(x_1,\dots, x_{n+1}):=x_1^2+\cdots+x_n^2-d^2(x_{n+1})$, we have
that $\Sigma=\Phi^{-1}(0)$, so that
\begin{equation} \label{eq-overbarN}
\overbar N:=\frac{\nabla\Phi}{\|\nabla\Phi\|}=\frac{\rho}{d}(x_1,\dots, x_n,-dd')
\end{equation}
is a Euclidean unit normal to $\Sigma$.
With this orientation, the Euclidean
mean curvature of $\Sigma$ is given by
(cf. proof of~\cite[Corollary 13.37]{andrewsetal}\footnote{Notice that, in~\cite{andrewsetal},
the mean curvature is nonnormalized, and the second fundamental form differs from ours by a sign.})
\begin{equation} \label{eq-overbarH}
\overbar H=\frac{\rho}{n}(d''\rho^2-n+1).
\end{equation}

Together with~\eqref{eq-MCrelation}, equations~\eqref{eq-overbarN}
and~\eqref{eq-overbarH} yield
\[
H=\frac{\rho x_{n+1}}{n}\left( \rho^2d''-\frac{n-1}{d}\right)-\rho d'.
\]
Also, for any $p=(x_1,\dots, x_{n+1})\in\Sigma$, one has
\[
\langle N,p\rangle=\rho\left(\frac{d}{x_{n+1}}-d'\right),
\]
and these two last equalities imply that
$\langle N,p\rangle=H$ is equivalent to~\eqref{eq-lemmahorizontalgraph}.
\end{proof}

Now, we can construct the translating catenoids in $\h^{n+1}$ just as
in~\cite{dLRS}. More specifically, we choose suitable solutions to the differential
equations~\eqref{eq-EDOPsi} and~\eqref{eq-lemmahorizontalgraph},
and then obtain the translating catenoid by gluing the corresponding horizontal
and vertical rotational graphs; see Figure~\ref{fig-translatingcatenoid}.
Moreover, as one can immediately see from the
results on the solutions of these equations obtained in~\cite{dLRS} for $n=2$,
their qualitative behavior is independent of the dimension $n\ge 2$.
Consequently, the reasoning in the proof of~\cite[Theorem 3.10]{dLRS}
can be employed to establish the following existence result.

\begin{figure}[htbp]
\includegraphics[width=\textwidth]{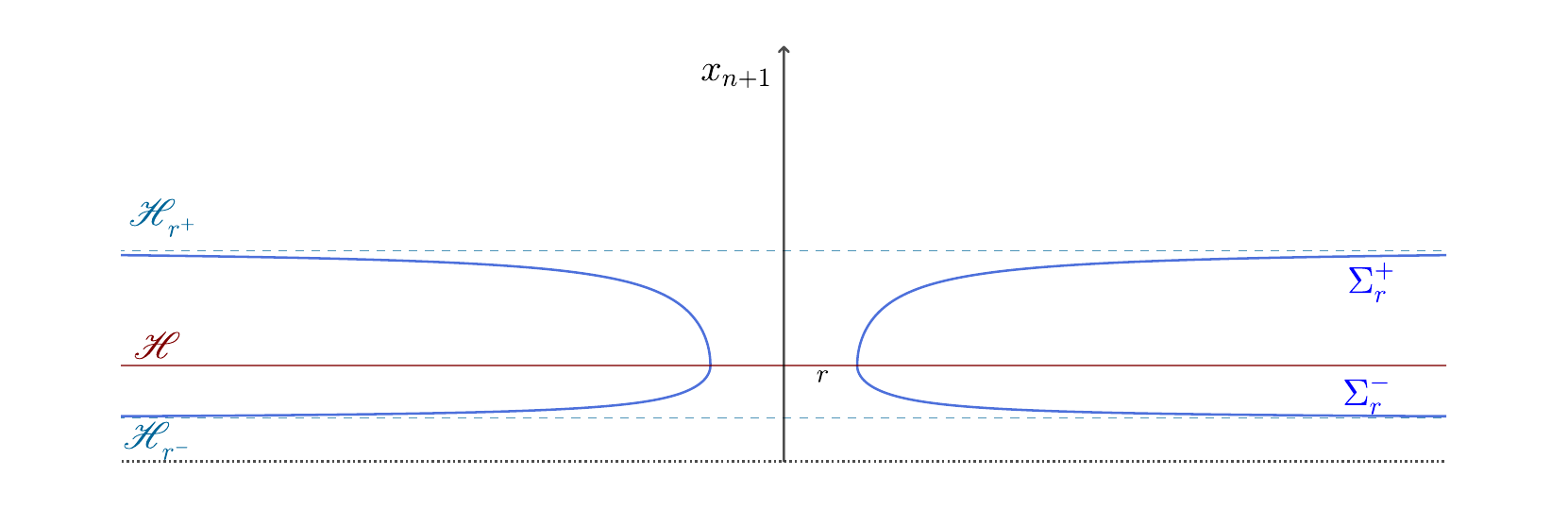}
\caption{\small {{The profile curve of a translating catenoid $\Sigma_r$
in $\h^{n+1}$ between (and asymptotic to)
two horospheres $\mathscr H_{r^-}$ and
$\mathscr H_{r^+}$. $\Sigma_r\setminus \mathscr H$ decomposes
as two vertical graphs $\Sigma_r^-$ and $\Sigma_r^+$ over the complement
of the Euclidean ball of radius $r$ centered at the rotation
axis $x_{n+1}$ in the horosphere $\mathscr H$.}}}
\label{fig-translatingcatenoid}
\end{figure}

\begin{theorem} \label{th-translatingcatenoids}
Given a horizontal horosphere $\mathscr H$,
there exists a one-parameter family
$\mathscr C:=\{\Sigma_{r}\mid r>0\}$
of noncongruent, properly embedded rotational
annular translators in $\h^{n+1}$ (to be called {\em
translating catenoids}). For
each $r>0$ (called the {\em neck size} of
$\Sigma_r\in \mathscr C$), it holds:
\begin{itemize}[parsep=1ex]
\item[\rm i)] $\Sigma_{r}$ is contained in a slab determined by
two horospheres $\mathscr H_{r^-}$ and $\mathscr H_{r^+}.$
In particular, the asymptotic boundary of $\Sigma_r$ is the
point at infinity of the horosphere $\mathscr H$.
%%%%%%%%%%%
\item[\rm ii)] $\Sigma_r$ is the union of two vertical graphs $\Sigma_{r}^-$ and
$\Sigma_{r}^+$ over the complement of the Euclidean
$r$-ball\, $\mathcal B_r$ centered at the rotation axis
in the horosphere $\mathscr H$.
%%%%%%%%%%%%%%
\item[\rm iii)] The graphs $\Sigma_{r}^-$ and $\Sigma_{r}^+$ lie in
distinct connected components of \,$\h^{n+1}~-~\mathscr H$
with common boundary the $r$-sphere that bounds $\mathcal B_r$ in
$\mathscr H,$ being $\Sigma_{r}^-$
asymptotic to $\mathscr H_{r^-}$ and $\Sigma_{r}^+$
asymptotic to $\mathscr H_{r^+}.$
\end{itemize}

In addition, when $r\to 0$ or when $r\to \infty$,
both $r^+$ and $r^-$ converge to 1 and
the limiting behavior of $\Sigma_r$ is
as follows:

\begin{itemize}[parsep=1ex]
\item[\rm iv)] As $r\to 0,$ $\Sigma_r$ converges
(on the $C^{2,\alpha}$-norm, on compact sets
away from $\{x_{n+1}=0\}$) to a double copy of
$\mathscr H.$

\item[\rm v)] As $r\to +\infty,$
$\Sigma_r$ escapes to infinity.
\end{itemize}
\end{theorem}

\begin{remark}
In Section~3.3 of~\cite{dLRS}, a tangency principle for translators to MCF
in $\h^{3}$ was established, as a consequence of the maximum principle for quasilinear
elliptic operators. Since the maximum principle is valid in all dimensions,
the tangency principle also holds in $\mathbb{H}^{n+1}$.
This fact, together with the asymptotic properties
of the family $\mathscr C$ presented in Theorem~\ref{th-translatingcatenoids},
implies that
there is no properly immersed
translator to {\rm MCF} in $\h^{n+1}$ which is contained in
the convex side of a geodesic cylinder with vertex at the origin; see the proof
of~\cite[Theorem 3.25]{dLRS}.
In particular, there is no closed (i.e., compact without boundary)
translator to {\rm MCF} in $\h^{n+1}$.
\end{remark}

\section{Stability of horospheres}\label{sec-stability}
In this section, we show that horospheres in $\h^{n+1}$
are dynamically stable as graphical solutions to MCF. More precisely, we have the following result.

\begin{theorem} \label{thm-stability}
Let $G(x,t)=e^{v(x,t)}x$, $x\in\mathcal S,\ t\geq0$,
be a graphical translating soliton
to {\rm MCF} in $\h^{n+1}$ whose initial condition
$G(\mathcal S\times\{0\})$ is a horizontal horosphere.
Set $v_0:=v(.\,, 0)$ and assume that
$u_0$ is a locally Lipschitz function on $\mathcal S$ such that
\begin{equation} \label{eq-condition}
\lim_{x\to\partial\mathcal S}|u_0(x)-v_0(x)|=0.
\end{equation}
Under these conditions, for any graphical solution $u=u(x,t)$
to~\eqref{eq-MCFequation} defined for all $t\ge 0$ and
with initial condition $u_0$, one has
that $u_t-v_t$
converges uniformly to 0 as $t\to+\infty$.
\end{theorem}

Theorem~\ref{thm-stability} constitutes the hyperbolic version
of~\cite[Theorem 1.1]{schulze}, which is set in Euclidean space.
It settles the stability of the graphical solution
$v(x,t)$ if we regard $u_0$ as a perturbation of $v_0$
with $C^0$ decay at infinity.
Indeed, as $t\to+\infty$, $u(x,t)$  becomes asymptotic
to the very solution $v(x,t)$ that was perturbed initially.
We also note that the notation $x\to\partial \cS$ means that
the variable $x$ escapes any compact set in $\cS$.

Next, we use
White's avoidance principle~\cite{white}
and the strong maximum principle for parabolic equations (see, for
instance,~\cite[Theorem 2.1.1 and Corollary 2.1.2]{mantegazza})
to prove Theorem~\ref{thm-stability}.
In the application of the avoidance principle,
we use the translating catenoids of
Theorem~\ref{th-translatingcatenoids} as barriers, which distinguishes
our proof from the proof in~\cite{schulze}.

\begin{proof}[Proof of Theorem~\ref{thm-stability}]
Let $v$, $v_0$ and $u_0$ be as stated. By Theorem~\ref{thm-unterberger},
there exists a function
$u\colon \mathcal{S}\times[0,\infty)\to \R$ such that
$F(x,t) = e^{u(x,t)}x$ is a graphical solution to MCF with
$u(x,0) = u_0(x)$ for all $x\in \cS$.
Let $\omega\colon \cS\times[0,\infty)\to \R$ be given by
$\omega(x,t) = u(x,t)-v(x,t)$; we need to show that
$\omega_t = \omega(\cdot,t)$
converges to zero uniformly when $t\to \infty$.

For any $r>0$, denote by $\mathscr C_r$ the closed, mean convex region
bounded by the rotational hyperbolic
cylinder of $\h^{n+1}$ of radius $r$ about the $x_{n+1}$-axis and let
$\Omega_r = \cS \cap \mathscr C_r$.
We also set the maps
$$F_0(x)=e^{u_0(x)}x \quad\text{and}\quad G_0(x)=e^{v_0(x)}x,
\,\,\, x\in\mathcal S,$$
and, througout the proof,
use the notation $u_t(x) = u(x,t)$, $v_t(x) = v(x,t)$.
Our first argument is to show that
$\lim_{x\to \partial \cS}\omega_t(x) = 0$, with a uniform
decay for all $t\geq0$.

\begin{claim}[Convergence in space]\label{cl:convergenceinspace}
For any $\epsilon >0$ there exists $R>0$ such that, for any $t\geq 0$,
\begin{equation} \label{eq-epsilon02}
\abs{\omega_t(x)} =|u_t(x)-v_t(x)|<\epsilon
\,\,\, \forall x\in\mathcal S-\Omega_{R}.
\end{equation}
\end{claim}
\begin{proof}[Proof of Claim~\ref{cl:convergenceinspace}]
Given $\epsilon >0$,
it follows from~\eqref{eq-condition} that there exists
$R_1>0$ for which
\begin{equation} \label{eq-epsiln00}
|u_0(x)-v_0(x)|<\epsilon/2 \,\,\, \forall x\in\mathcal S-\Omega_{R_1}.
\end{equation}
Since
$\abs{u_t(x)-v_t(x)}$ is the hyperbolic length of the
orthogonal projection of the segment joining $F(x,t)$ and $G(x,t)$
over the $x_{n+1}$-axis of $\R^{n+1}_+$,~\eqref{eq-epsiln00}
implies that $F_0(\mathcal S-\Omega_{R_1})$
is contained in the slab
\begin{align*}
\Lambda_\epsilon = \{(x_1,\,x_2,\,\ldots,\,x_{n+1})\in \hn{n+1}\mid
e^{-\epsilon/2}<x_{n+1}<e^{\epsilon/2}\}.
\end{align*}
Let, for any $h\in \R$,
$\mathscr H_h$ denote the horizontal horosphere at height $e^h$.
Then, $F_0(\cS-\Omega_{R_1})$ is between $\mathscr H_{-\epsilon/2}$ and
$\mathscr H_{\epsilon/2}$,
with a positive distance from these horospheres.
We also consider
\begin{align*}
\Lambda_\epsilon^+ &= \{(x_1,\,x_2,\,\ldots,\,x_{n+1})\in \hn{n+1}\mid
e^{\epsilon/2}<x_{n+1}<e^{\epsilon}\}\\
\Lambda_\epsilon^- &= \{(x_1,\,x_2,\,\ldots,\,x_{n+1})\in \hn{n+1}\mid
e^{-\epsilon}<x_{n+1}<e^{-\epsilon/2}\}
\end{align*}
the slabs of width $\epsilon/2$, adjacent to $\Lambda_\epsilon$, with
$\Lambda_\epsilon^+$ above and $\Lambda_\epsilon^-$ below
$\Lambda_\epsilon$.

By Theorem~\ref{th-translatingcatenoids}, there
exist two translating catenoids $\Sigma^-$ and $\Sigma^+$ in $\h^{n+1}$,
both disjoint from $\mathscr C_{R_1}$, with
$\Sigma^-\subset \Lambda^-_\epsilon$ and asymptotic
to $\mathscr H_{-\epsilon/2}$ and $\Sigma^+\subset\Lambda_\epsilon^+$ and
asymptotic to
$\mathscr H_{\epsilon/2}$, see Figure~\ref{fig-barriers}.
By construction, both $\S^+$ and $\S^-$
are complete translators to MCF in $\hn{n+1}$, each of which
a positive distance from $F_0(\mathcal S)$.

Let, for $t\in \R$, $\Gamma_t$ be the
hyperbolic translation of
$\hn{n+1}$ on the direction of the $x_{n+1}$-axis,
given in coordinates by~\eqref{eqtranslhyp}.
Then, White's avoidance principle~\cite[Theorem~1]{white}
implies that, for all $t\geq0$, $F(\cS\times\{t\})$ stays a positive
distance from both $\Gamma_t(\S^-)$ and $\Gamma_t(\S^+)$.
In particular, if $R>0$ is such that $\mathscr C_R$ intersects
both $\S^-$ and $\S^+$ (in particular, $R>R_1$),
for any $x\in \cS\setminus \mathscr C_R$ and any
$t\geq 0$ we have that $F(x,t)$ lies in the slab
bounded between the horospheres
$\Gamma_t(\mathscr H_{-\epsilon}) =\mathscr H_{t-\epsilon}$
and $\Gamma_t(\mathscr H_{\epsilon}) = \mathscr H_{t+\epsilon}$,
which proves~\eqref{eq-epsilon02}.
\end{proof}

\begin{figure}
\centering
\includegraphics[width=\textwidth]{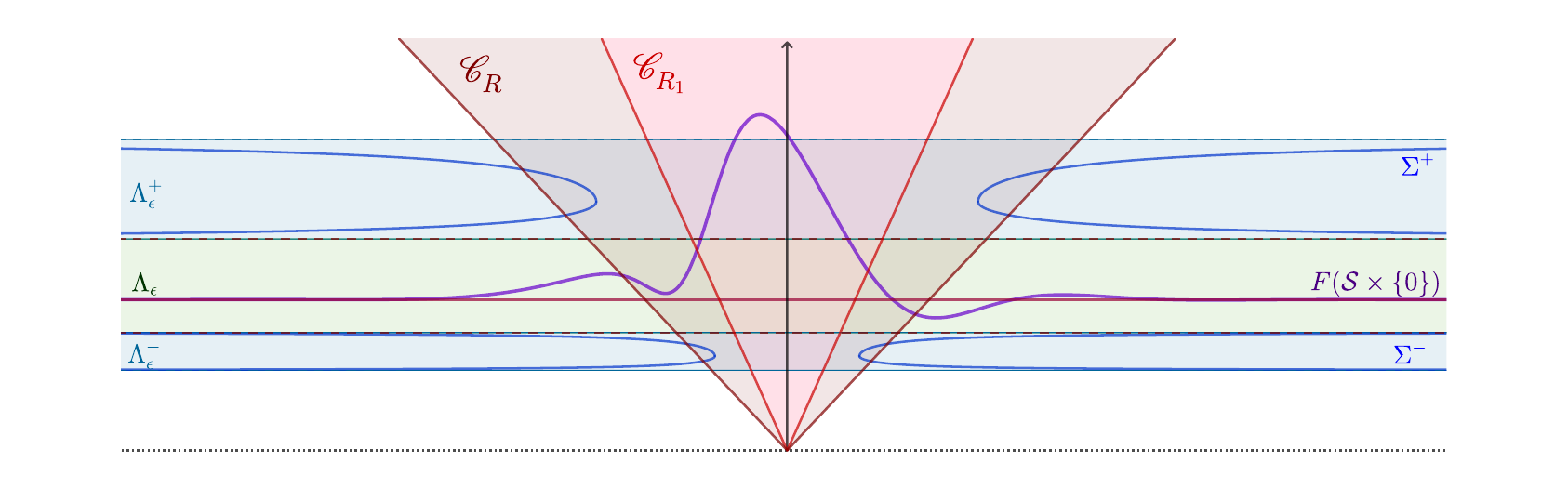}
\caption{\small Translating catenoids acting as barriers:
$\S^+\subset \Lambda_\epsilon^+$ and $\S^-\subset \Lambda_\epsilon^-$
are a positive distance from the graph
$F(\mathcal S\times\{0\})$, whose part outside $\mathscr C_{R_1}$
lies in the slab $\Lambda_\epsilon$.
As the surfaces flow under the MCF,
$F_t(\mathcal{S}\setminus \mathscr C_R)$ remains
between the translating catenoids
$\Gamma_t(\Sigma^-)$ and $\Gamma_t(\Sigma^+)$.}
\label{fig-barriers}
\end{figure}

Next, we argue as in~\cite[Lemma 4.2]{schulze} to
show that $\omega_t$ converges, in the $C^0$ norm,
to zero when $t\to \infty$.
For $\epsilon >0$, let $R$ be given by Claim~\ref{cl:convergenceinspace}.
We will use $R$ to show that there exists some $T>0$
such that
$\abs{\omega_t(x)} <\epsilon$ for all $x\in \cS$ and all $t\geq T$.

First, a standard argument
(for instance, as in the proof of~\cite[Theorem 17.1]{GT})
implies that $\omega$ satisfies a linear parabolic equation. Indeed,
let $\mathcal{Q}$ be the quasi-linear elliptic operator
from~\eqref{eqopQ}. Then, $u$ and $v$ satisfy
\begin{align*}
\frac{\partial u}{\partial t}(x,t)=\mathcal Q(u),\quad
\frac{\partial v}{\partial t}(x,t)=\mathcal Q(v).
\end{align*}
For $\theta\in[0,1]$, let $v^\theta = \theta u+(1-\theta)v
= v + \theta \omega$, so $v^0 = v$ and $v^1 = u$,
while $\frac{d}{d\theta}v^\theta = \omega$.
Then,
\begin{align}\label{eqintermediate1}
\frac{d \omega}{d t} & = \mathcal Q(u)-\mathcal Q(v)
 = \int_0^1 \frac{d}{d\theta} \mathcal Q (v^\theta)d\theta.
\end{align}
Recall that
$\mathcal Q (v^\theta) =
Q(x,Dv^\theta_t,D^2v^\theta_t)$ and let $Q_{p_i}$
denote the derivative of
$Q$ in the $i$-th direction of $p\in \R^{n+1}$ and
$Q_{r_{ij}}$ denote the derivative of $Q$ in the $(i,j)$ entry
of the matrix $r\in \R^{(n+1)\times(n+1)}$.
It follows from~\eqref{eqintermediate1} that
\begin{align}\label{eqparaboliclin}
\frac{d \omega}{d t}(x,t) & = \sum_{i,j} a^{ij}(x,t)D_{ij}(\omega_t)(x)
+\sum_i b^i(x,t) D_i(\omega_t)(x),
\end{align}
where
\begin{align*}
a^{ij}(x,t)
&= \int_0^1 Q_{r_{ij}}(x,Dv^\theta_t,D^2v^\theta_t)d\theta,&
b^{i}(x,t)
&= \int_0^1 Q_{p_{i}}(x,Dv^\theta_t,D^2v^\theta_t)d\theta
\end{align*}
are coefficients that depend uniquely on the coordinates $x$ and
$t$.

For $t\ge 0$, define
$\Upsilon_t:=\{x\in\mathcal S\,;\, \omega_t(x)\ge \epsilon\}$.
Assume that $\Upsilon_{t_0}\neq \es$ for some $t_0>0$.
Then,
there exists $\delta>0$ such that $\Upsilon_t\neq \es$
for all $t\in(t_0-\delta,t_0+\delta)$. In particular, for any such
$t$, $\omega_t$ attains
its maximum in some point of $\Upsilon_t$, and
$\Upsilon_t\subset\Omega_{R}$, by~\eqref{eq-epsilon02}.
Hence, together with the fact that
$\omega$ satisfies~\eqref{eqparaboliclin}, the strong maximum
principle~\cite[Theorem 2.1.1]{mantegazza}
applies to show that the function
$t\in (t_0-\delta,t_0+\delta)\mapsto {\rm max}_\cS u_t$
(which is continuous, since $\lim_{x\to \partial \cS} \omega_t(x) = 0$,
and locally Lipschitz)
is nonincreasing. Thus, if
$\Upsilon_{t_0} = \emptyset$ for some $t_0>0$, then $\Upsilon_{t} =
\emptyset$ for all $t\geq t_0$.

Next, we will show that $\Upsilon_{t_0} = \es$ for some $t_0 >0$.
Arguing by contradiction, we
assume that $\Upsilon_t\neq \emptyset$ for all $t>0$,
so the function ${\rm max}(\omega_t)$ is nonincreasing with
$t$ and satisfies ${\rm max}(\omega_t)\geq \epsilon$.
For $n\in\N$, let $\omega^n\colon \cS\times[-n,\infty)\to \R$
be defined by $\omega^n(x,t) = \omega(x,t+n)$.
Then $\{\omega^n\}_{n\in\N}$ is uniformly bounded and, after passing
to a subsequence, it converges
smoothly to
some function $\gamma$ that also satisfies
\begin{align}\label{eqlimitgamma}
\abs{\gamma(x,t)}<\epsilon,\quad x\in \cS - \Omega_{R}\text{ and }
t\geq0.
\end{align}

We claim that the graph of the function $\phi = \gamma+v$
evolves via MCF. To see this,
let $\phi^n = \omega^n + v$. Then, for $x\in \cS$ and $t\geq 0$,
\begin{align*}
\phi^n(x,t) & = u(x,t+n)-v(x,t+n)+v(x,t) = u(x,t+n)-n,
\end{align*}
where the second equality follows from the fact that the graph
of $v$ is a translating soliton to MCF. In particular,
the graph of $\phi^n$ evolves by MCF and it follows
from~\eqref{eq-MCFequation02} that
\begin{align*}
\frac{d\phi^n}{dt} = \cQ(\phi^n) \quad \Longrightarrow \quad
\frac{d\phi}{dt} = \cQ(\phi),
\end{align*}
since, for a fixed $t$,
$\lim \phi^n(\cdot,t) = \phi(\cdot,t)$
in the $C^{2,\alpha}$ norm in compacts of $\cS$.
Once again, we may show that $\gamma = \phi - v$ satisfies
the hypothesis of the strong maximum principle for linear
parabolic equations. However, the defining properties of
$\gamma$ give that ${\rm sup}_{x\in \cS}\gamma(x,t)$ does not
depend on $t$, so we may apply~\cite[Corollary~2.1.2.]{mantegazza}
to obtain that $\gamma$ is in fact constant, thus $\gamma \equiv 0$
by~\eqref{eqlimitgamma}. On the other hand, the
assumption that, for all $t$, ${\rm max}(\omega_t) \geq \epsilon$,
attained in some point of the compact $\Omega_{R}$, gives that $\gamma$
cannot be identically zero, which is a contradiction
that proves that $\omega_t(x) <\epsilon$ for all $t$ sufficiently large.

The proof that
$v(x,t)-u(x,t) <\epsilon$ for all $x\in \cS$ and all
$t$ sufficiently large is analogous, after observing that $u$ is a
graphical solution to MCF if and only if $-u$ also is. This completes
the proof of the theorem.
\end{proof}

\section{Further developments}\label{secopenprob}

In this section, we extend the arguments
of~\cite[Section~3.2]{dLRS} to the context of $\hn{n+1}$, presenting
a family of parabolic invariant translators to MCF which are analogous
to the grim reapers of $\R^{n+1}$. After noticing
that this family also consists of graphical solutions
to MCF, we raise two questions concerning stability.

\subsection{Grim reapers}
We now proceed to construct in $\h^{n+1}$ the analogous of
the translators of $\h^{3}$ called \emph{grim reapers.}
Such translators are invariant by a
group of {\em parabolic} isometries of $\h^{n+1}$, that is,
those that leave invariant
families of parallel horospheres .

As in the case $n=2$, the grim reapers obtained here are
horizontal cylinders over entire graphs on $\R$ which are contained in
a vertical totally geodesic plane
of \,$\h^{n+1}$. More precisely,
such a hypersurface is parameterized by a map
$X\colon\R^n\to\h^{n+1}$ of the form
\begin{equation*}
X(x_1,\dots, x_{n-1},s)=(x_1,\dots, x_{n-1},s,\phi(s)), \,\, (x_1,\dots, x_{n-1},s)\in\R^n,
\end{equation*}
where $\phi$ is a smooth positive function on $\R.$
We call $\Sigma:=X(\R^n)$ the \emph{parabolic cylinder determined by}
$\phi.$

Reasoning as in Section~\ref{sec-translators} (see also the proof
of~\cite[Lemmas 3.18 and 3.19]{dLRS}), one obtains the
following result.

\begin{lemma} \label{lem-parabolicODE01}
A parabolic cylinder determined by a positive smooth function $\phi$ is
a translator to {\rm MCF} in \,$\h^{n+1}$ if and only
if {$\phi=\phi(s)$} is a solution to
the second order {\rm ODE:}
\begin{equation} \label{eq-EDOparabolic}
\phi''=-\phi'(1+(\phi')^2)\frac{ns}{\phi^2}\cdot
\end{equation}
Moreover, any solution $\phi$ is defined for all $s\in \R$
and, assuming
$\lambda = \phi'(0)\geq 0$, it satisfies (see Figure~\ref{figGR}):
\begin{enumerate}[label={\rm (\roman*)}]
\item $\phi$ is constant if $\lambda = 0$;
\item \label{item2ofthelemma} $\phi$ is increasing, convex in $(-\infty,0)$ and
concave in $(0,+\infty)$ if $\lambda > 0$;
\item There exist $\lambda^-\in (0,\phi(0))$ and $\lambda^+\in (\phi(0),+\infty)$ such that
\begin{align*}
\lim_{s\to - \infty}\phi(s)=\lambda^-, \quad
\lim_{s\to + \infty}\phi(s)=\lambda^+.
\end{align*}
\end{enumerate}
\end{lemma}

\begin{figure}
\centering
\includegraphics[width=\textwidth]{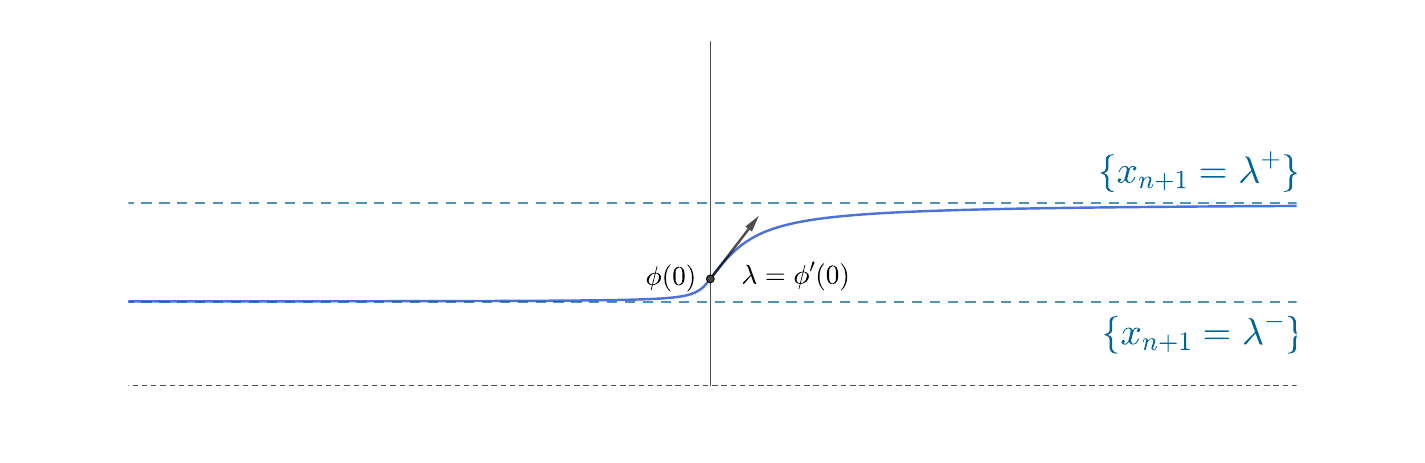}
\caption{\small\label{figGR} The profile curve of a solution
$\phi$ to~\eqref{eq-EDOparabolic} with $\lambda = \phi'(0) >0$.}
\end{figure}

It follows from Lemma~\ref{lem-parabolicODE01} that we can mimic the proof
of~\cite[Theorem 3.20]{dLRS} to obtain the following result.

\begin{theorem} \label{th-grimreapers}
There exists a one-parameter family
\begin{equation*}
\mathscr G:=\{\Sigma_{\lambda}\mid \lambda\in[0,+\infty)\}
\end{equation*}
of noncongruent complete translators (to be called \emph{grim reapers})
which are horizontal parabolic cylinders
generated by the solutions of~\eqref{eq-EDOparabolic}.
$\Sigma_0$ is the horosphere $\mathscr H$ at height one, and
for $\lambda>0,$ each $\Sigma_{\lambda}\in\mathscr G$ is an entire
vertical graph over $\mathscr H$ which is
contained in a slab determined by two
horospheres $\mathscr H_{\lambda^-}$ and $\mathscr H_{\lambda^+}.$
Furthermore, there exist open sets $\Sigma_{\lambda}^-$ and $\Sigma_{\lambda}^+$ of $\Sigma_\lambda$
such that $\Sigma_{\lambda}^-$ is
asymptotic to $\mathscr H_{\lambda^-}$, $\Sigma_{\lambda}^+$ is asymptotic to $\mathscr H_{\lambda^+},$
and $\Sigma_\lambda={\rm closure}\,(\Sigma_{\lambda}^-)\cup{\rm closure}\,(\Sigma_{\lambda}^+).$
\end{theorem}

\begin{remark}\label{rem-soliton}
Each element $\S_\lambda$ in the family $\mathscr G$
provided by Theorem~\ref{th-grimreapers} defines a graphical
solution to MCF. Indeed, if $\lambda = 0$, the family of translating
horospheres can be explicitly parameterized by the function
$u(x,t) = t - \log(x_{n+1})$. When $\lambda >0$, one just needs
to verify that the initial condition $\Sigma_\lambda$ is
graphical in the sense of~\eqref{eq-map}, since hyperbolic translations
do not change this property. Although
the expression of the graphing function of $\S_\lambda$
is not explicit, item~\ref{item2ofthelemma} of
Lemma~\ref{lem-parabolicODE01} implies that
any half-line in $\R^{n+1}_+$ with its endpoint
at the origin meets $\S_\lambda$ exactly once.
\end{remark}

\subsection{Open problems.}
The two main steps in the proof of Theorem~\ref{thm-stability}
are the uniform convergence in space of the functions $\omega_t$,
and the convergence in time of the family $\{\omega_t\}$. We note
that the proof of the time convergence, which is the statement of
the theorem, is independent of the initial
translating soliton considered, as long as it is a graphical solution
that satisfies the conclusion of Claim~\ref{cl:convergenceinspace}.
Since each member $\S_\lambda$ of the family $\mathscr G$ presented
in Theorem~\ref{th-grimreapers} is a translator and also
a graphical solution to MCF, for which $\Sigma_0$ is dynamically stable,
we pose the following question:
\begin{center}
{\em Are grim reapers
in $\hn{n+1}$ dynamically stable graphical solutions
to} MCF?
\end{center}

This question remains open in the Euclidean
setting as well. Indeed, it is unknown whether
the grim reaper in $\R^{n+1}$
is stable or not as a graphical translator to MCF. Thus,
this raises the question:
\begin{center}
{\em Is a grim reaper in $\R^{n+1}$ a
dynamically stable graphical solution to} MCF?
\end{center}

Finally, we ask whether the Lipschitz continuity assumption
in Theorem~\ref{thm-unterberger} can be relaxed to continuity,
as in~\cite{schulze,EH}.

\end{document}